\documentclass{amsart}
\usepackage{amsmath}
\usepackage{amssymb}
\usepackage{mathtools}
\usepackage{amsthm}
\usepackage{float}
\usepackage{pgf, tikz}
\usetikzlibrary{arrows, automata}

\newtheorem{theorem}{Theorem}

\newtheorem{lemma}{Lemma}

\title{Schur-like numbers and a lemma of Shearer}
\author{Tomasz Kościuszko}
\begin{document}
\begin{abstract}
Suppose that each number $1,2,\cdots, N$ has one of $n$ colours assigned. We show that if there are no monochromatic solutions to the equation $x_1+x_2+x_3=y_1+y_2$, then $N=O(\sqrt{n!})$, improving upon a result of Cwalina and Schoen. Further, a stronger bound of $N=O(\sqrt{(n-k)!})$, where $k\gg\frac{\log n}{\log\log n}$ is shown for colourings avoiding solutions to the equation $x_1+x_2+\cdots+x_{12}=y_1+y_2+\cdots+y_{9}$. Finally, some remarks on other equations are presented.
\end{abstract}
\maketitle
\section{Introduction}
Schur's Theorem~\cite{schur} is one of the oldest theorems in the area, which today we call additive combinatorics. Despite its remarkably simple statement, one cannot overestimate the influence it had in the past century. \begin{theorem}\label{thm:schur}(Schur)
Let $N$ be a positive integer and suppose that each of the numbers $1,2,\cdots, N$ has one of $n$ colours assigned to it. If that there are no monochromatic solutions to the equation $x+y=z$, then $N\leq \lfloor n!\cdot e\rfloor.$    
\end{theorem}
Proved in 1917, Schur's Theorem precedes two other major results in combinatorics - Van der Waerden's Theorem and Ramsey's Theorem.
Van der Waerden's Theorem may seem similar because the question here is avoiding monochromatic arithmetic progressions. The first non-trivial case concerns arithmetic progressions of length~3, which correspond to monochromatic solutions of the equation $x+y=2z$. This equation is invariant, which means that the sum of the coefficients is equal to~0. Much has been discovered about such equations (see \cite{km}, \cite{schoen}, \cite{szem} for recent developments). Roth's Theorem and Szemeredi's Theorem tell us that presence of solutions can be deduced by looking solely at the densities of colours involved. If a set, or a colour, is large enough, then it is bound to contain the desired solution.

The situation is much different in the case of Schur's equation, $x+y=z$. It is not invariant, which means that shifting a set could introduce potential solutions, whereas the density remains the same. There also exist sets of large density without Schur triples, for example, the set of odd numbers. That makes the problem all the more intriguing, as we are forced to understand how different colours could interact in order to cover as many consecutive integers as possible.

We denote the least number $N$ for which a colouring satisfying the requirement of Theorem~\ref{thm:schur} is not possible as $S(n)$ and call it Schur's number.
Ramsey's Theorem about avoiding monochromatic triangles in a colouring of a complete graph turns out to be closely related to the problem of determining $S(n)$.
In fact, any upper bound in Ramsey's Theorem is an upper bound in Schur's Theorem and any lower bound for Schur's number is a lower bound for the Ramsey number $r_n(3)$.

The link between the two theorems is a graph in which edges correspond to the set $A-A$, where $A$ is a set containing numbers of the same colour.
This connection will be explained in detail when we recall the graph theoretic proof of Schur's Theorem in Section~2. The same connection will be exploited in the proofs of Theorems~\ref{thm:schur_like} and~\ref{thm:long}. \\ 
Concerning the lower bound, Schur proved that $S(n) > (3^n+1)/2$.
It is worth noting that the greedy strategy also gives an exponential lower bound and works for all equations which are not invariant. There, for suitably chosen $\alpha>0$ we use colour $i$ to cover the interval $[\alpha^{i-1}, \alpha^i)\cap\mathbb{N}$. For example $\alpha=2$ works for the equation $x+y=z$ and $\alpha=1.5$ for the equation $x_1+x_2+x_3=y_1+y_2$.\\
The bounds for Schur's number have been improved slightly (see \cite{abbot}, \cite{chung}, \cite{exoo}, \cite{fredricksen}, \cite{wan}) and currently stand at
$$(3.19)^n\ll S(n)\ll \lfloor(e-\frac{1}{24})n!\rfloor.$$
Cwalina and Schoen, in their paper~\cite{cs}, proposed an extension to Schur's problem. They introduced schur-like numbers, $S_k(n)$ is the smallest number $N$ for which any $n$-colouring of $[N]:=\{1,2,\cdots, N\}$ must contain a monochromatic solution to the equation 
$$x_1+x_2+\cdots x_{k+1}= y_1+y_2+\cdots y_k.$$
Clearly $S_i(n)\leq S_j(n)$ for $i>j$ and $S_1(n)=S(n).$ They proved the following result about $S_2(n)$. 
\begin{theorem}\label{thm:cs}(Cwalina-Schoen)
Let $N$ be a positive integer and suppose that each of the numbers $1,2,\cdots, N$ has one of $n$ colours assigned to it. Suppose that there are no monochromatic solutions to the equation
$$x_1+x_2+x_3=y_1+y_2.$$
Then we have $N=O(n^{-c\log n/\log\log n}n!)$,
for a small absolute constant $c$.
\end{theorem}
Their second impressive result was to generalise Schur's Theorem to all non-invariant, regular equations. They use deep results about Bohr sets and a complicated iterative procedure in order to have induction on the colours, as in Schur's proof. Prior to them, it was however highly non-obvious that an inductive proof can be given for a general equation. Their result is given below. In the statement, a regular equation is an equation which contains an invariant equation. It is a result of Rado\cite{rado} that only such equations are worth considering, for non-regular ones there exists a colouring of $\mathbb{N}$ with number of colours depending on the equation, which avoids solutions to it.
\begin{theorem}\label{thm:general}(Cwalina-Schoen)
Let $a_1 x_1 + a_2 x_2 +\cdots a_k x_k = 0$ be a regular non-invariant equation with integer coefficients. Let $N$ be a positive integer and suppose that each of the numbers $1,2,\cdots, N$ has one of $n$ colours assigned to it. Suppose that there are no monochromatic solutions to this equation.
Then we have
$$N\ll 2^{O(n^4\log^4 n)}.$$
Furthermore, if the equation contains an invariant equation with at least four variables, then $N\ll 2^{O(n^3\log^5 n)}.$ Finally, if the equation contains an invariant equation of genus at least $2$, then $N\ll 2^{O(n^2\log^5 n)}.$
\end{theorem}
The first result of this paper improves upon the bound on $S_2(n)$ mentioned in Theorem~\ref{thm:cs}.
\begin{theorem}\label{thm:schur_like}
Let $N$ be a positive integer and suppose that each of the numbers $1,2,\cdots, N$ has one of $n$ colours assigned to it. Suppose that there are no monochromatic solutions to the equation
$$x_1+x_2+x_3=y_1+y_2.$$
Then we have $N=O(\sqrt{n!})$.    
\end{theorem}
Our most important tool is a lemma by Shearer, which allows us to find a large independent set in a graph free of triangles. The lemma will be stated in the next section. We also use the final result from the paper of Shearer~\cite{shearer} to prove a stronger bound, however for a much longer equation.
\begin{theorem}\label{thm:long}
Let $N$ be a positive integer and suppose that each of the numbers $1,2,\cdots, N$ has one of $n$ colours assigned to it. Suppose that there are no monochromatic solutions to the equation
$$x_1+x_2+\cdots+ x_{12}=y_1+y_2+\cdots+y_{9}.$$
Then we have $N=O(\sqrt{(n-k)!})$, where $k\geq c\frac{\log n}{\log\log n}$ for some small constant $c$.   
\end{theorem}
It would seem that the same bound should hold for some $S_k(n)$ for sufficiently large $k$, however we were not able to prove it.\\
At the end of the paper we address two equations different to the ones from Theorem~\ref{thm:schur_like} and Theorem~\ref{thm:long}. One of them is $$x_1+x_2+\cdots+x_l = y,$$ seemingly similar to Schur's equation, however it is not clear whether the same bounds should hold. To our knowledge, the only existing bound is $$N=O(\exp(n^4\log^4 n))$$ coming from Theorem~\ref{thm:general} from the paper of Cwalina and Schoen. By a repeated application of Schur's Theorem we show a bound closer to $S(n)$. 
\begin{theorem}\label{thm:imbalanced}
Let $l\geq 3$ be an integer. Let $N$ be a positive integer and suppose that each of the numbers $1,2,\cdots, N$ has one of $n$ colours assigned to it. Suppose that there are no monochromatic solutions to the equation $x_1+x_2+\cdots + x_l=y$. Then we have $N\leq e(n\lceil\log_2 l\rceil)!$.
\end{theorem}
Finally, we use a similar argument to show a bound for the equation $$x_1+2x_2+x_3=y_1+y_2.$$ One of the coefficients is not equal to $\pm1$, which is a departure from all the previous cases. Again, the best known bound comes from Theorem~\ref{thm:general} and stands at ${N=O(\exp(n^2\log^5 n))}$, as $x_1+x_3=y_1+y_2$ is an equation of genus 2.
\begin{theorem}\label{thm:coefficients}
Let $N$ be a positive integer and suppose that each of the numbers $1,2,\cdots, N$ has one of $n$ colours assigned to it. Suppose that there are no monochromatic solutions to the equation $x_1+2x_2+x_3=y_1+y_2$. Then we have $N\leq e(2n)!$.
\end{theorem}
\section{Shearer's lemma and the consequences}
This section will be focused on the proof of Theorem~\ref{thm:schur_like}. It will follow the same basic idea as the classic graph theoretic proof of Schur's Theorem, which we briefly recall here.
\begin{proof}(Of Schur's Theorem~\ref{thm:schur})
Let $A_1,A_2,\cdots, A_n$ be sets containing numbers in each colour.
For each $A_i$ we construct a graph on $N$ vertices, where each vertex is identified with a number from $[N]$. Two vertices $v,w$ are connected if $|v-w|\in A_i$. Our proof will follow $n$ steps, where in each step we will construct a subset $X_i\subseteq [N]$. Initially we set $X_0=[N]$. We will prove inductively, that for every $X_i$, there exists $X_{i+1}\subseteq X_i$, such that $$|X_{i+1}|\geq \frac{|X_i|-1}{n-i}$$ and the complete graph with vertices in $X_{i+1}$ contains at most $n-(i+1)$ colours.\\
Let $v$ be an arbitrary vertex of $X_i$. The edges of the complete graph with vertices in $X_i$ have at most $n-i$ colours, which means that $v$ has at least $\frac{|X_{i}|-1}{(n-i)}$ outgoing edges in the same colour, without loss of generality call it $A_i$. Call the set of vertices connected to $v$ with this colour $X_{i+1}$. Let $s,t\in X_{i+1}$, we claim that $s$ and $t$ are not connected with an edge of colour $A_i$. If that was the case, then we would have a monochromatic triangle with vertices $v,s,t$. If $v_1 < v_2 < v_2$ is the increasing ordering of $v,s,t$, then $v_2-v_1, v_3-v_2, v_3-v_1\in A_i$. Then $(v_2-v_1)+(v_3-v_2)=(v_3-v_1)$ is a monochromatic solution to Schur's equation.\\
After $n$ steps we obtain a set $X_n$, which has size at most $1$ - it corresponds to a complete graph with edges free of all colours. We can show by induction that
$$|X_{n-m}|\leq m!\sum_{i=0}^{m}\frac{1}{i!}.$$
It is clearly true for $m=0$, for $m>0$ we have
$$|X_{n-m}|\leq m(|X_{n-m+1}|-1) \leq m (m-1)!\sum_{i=0}^{m-2}\frac{1}{i!}\leq m!\sum_{i=0}^{m}\frac{1}{i!}.$$
Thus we have proved $N=|X_0|\leq n!\sum_{i=0}^{n}\frac{1}{i!}\leq n!\cdot e$.
\end{proof}
We ideally would like to find a way to have sets $X_i$ larger with each step. In the case of Schur's Theorem there is no clear way how to do it, however if we consider the equation $x_1+x_2+x_3=y_1+y_2$, there is a way to construct a larger independent set.
Let us recall a lemma by Shearer from his work~\cite{shearer}.
\begin{lemma}\label{lem:shearer1}
Let $G$ be a triangle-free graph. Then the independence number $\alpha(G)$ of $G$ satisfies
$$\alpha(G)\geq \sum_{v\in G}d_1(v)/(1+d_1(v)+d_2(v)),$$
where $d_i(v)$ is the number of vertices at distance $i$ from $v$.
\end{lemma}
We refer the reader to the original paper for a probabilistic proof. This lemma was ultimately used by Lin~\cite{lin} to improve upper bound on the multi-colour Ramsey number for odd cycles. Our proof relies on the fact that cycles of length 5 and solutions to the equation $x_1+x_2+x_3=y_1+y_2$ turn out to behave similarly.
\begin{proof}(Of Theorem~\ref{thm:schur_like})
Let $A_1,A_2,\cdots, A_n$ be sets containing numbers in each colour.
We will construct $n+1$ sets $$X_0\supseteq X_{1}\supseteq\cdots\supseteq X_n,$$ where $X_0=[N]$. Each set $X_i$ will correspond to a complete graph with edges in at most $n-i$ colours. The edges in each of these graphs will correspond to the colours $A_j$ in the following way. Let $u,v\in X_i$, then the edge between $u$ and $v$ will have colour $A_j$ if $|v-u|\in A_j$.\\
Without loss of generality let $A_i$ be the largest colour in $X_i-X_i$, that is the one that creates the most edges between vertices corresponding to $X_i$. Let $G_i$ be the graph with vertices in $X_i$ and edges in the colour $A_i$. 
For a fixed element $a\in X_i$ and corresponding vertex $v_a\in G_i$, we will now construct two large independent sets. One of them is the larger of these two sets.
$$R_1:=\{v_{a+x}: x\in A_i, a+x\in X_i\}$$
$$R_2:=\{v_{a-x}: x\in A_i, a-x\in X_i\}$$
They are both independent as $(A_i-A_i)\cap A_i = \emptyset$. The larger one has at least $d_1(v_a)/2$ elements because $d_1(v_a)= R_1\cup R_2$.\\ The other set will be the largest out of the following three.
$$S_1=\{v_{a+x+y}: x,y\in A_i, a+x+y\in X_i\}$$
$$S_2=\{v_{a+x-y}: x,y\in A_i, a+x-y\in X_i\}$$
$$S_3=\{v_{a-x-y}: x,y\in A_i, a-x-y\in X_i\}$$
We know that $S_1,S_2$ and $S_3$ are all independent as $(2A_i-2A_i)\cap A_i = \emptyset$.\\ Moreover $d_2(v_a)\cup \{v_a\}=S_1\cup S_2\cup S_3$, therefore the largest out of the three sets has size at least $(1+d_2(v_a))/3$.\\ We therefore  have $$5\alpha(G_i)\geq 3\max(|S_1|,|S_2|,|S_3|) + 2\max(|R_1|,|R_2|) \geq1+d_1(v_a)+d_2(v_a).$$ 
Clearly $G_i$ is triangle free as any triangle would correspond to a solution to our equation. Therefore by Lemma~\ref{lem:shearer1} we have
$$\alpha(G_i)\geq \sum_{v\in G_i}d_1(v)/(1+d_1(v)+d_2(v)).$$
Bounding the denominator by $5\alpha(G_i)$ we get
$$5\alpha(G_i)^2\geq \sum_{v\in G_i}d_1(v)\geq |G_i|(|G_i|-1)/(n+1-i)$$ and so $$\alpha(G_i)\geq \frac{|G_i|}{3\sqrt{n+1-i}}.$$ We choose $X_{i+1}$ to be the largest independent set of the graph $G_i$. We have established that
$$|X_{i+1}|\geq \frac{|X_i|}{3\sqrt{n+1-i}}.$$
After $n$ iterations of the above argument we have a set $X_{n}$ for which $|X_n|\gg \frac{N}{\sqrt{n!}}$, but also the corresponding edges do not contain any colours so $|X_n|\leq 1$. We conclude that $N=O(\sqrt{n!})$.
\end{proof}
\section{Improving the first iteration}
Our next result is a slight strengthening on the previous result, given we consider a longer equation. We use a more powerful version of Lemma~\ref{lem:shearer1} for regular graphs. It comes from the paper of Shearer~\cite{shearer} and is an intermediate step in the proof of Theorem~2.
\begin{lemma}\label{lem:shearer2}
Let G be a graph on $N$ vertices of odd girth $2m+3$ or greater ($m\geq 2$). The independence number of this graph $\alpha(G)$ can be bounded as
$$\alpha(G)\geq \sum_{v\in G}\Bigl[\frac{d_1(v)2^{-(m-2)}}{1+d_1(v)+d_2(v)+\cdots +d_m(v)}\Bigr]^{1/(m-1)}.$$
\end{lemma}
The idea in the proof of Theorem~\ref{thm:long} is that we can use Lemma~\ref{lem:shearer2} to construct a large set without a colour, provided the colour is not too small. We can however do so only in the first step of the induction, where our graph is regular. Luckily, using the idea of Cwalina and Schoen, we can perform the first $k$ steps as it were, in parallel. A lemma from their paper will be useful in showing that there are enough large colours.
\begin{lemma}\label{lem:factorial}
    Suppose that $\{1\cdots, N\} = A_1\cup\cdots\cup A_n$ is a partition into sum-free sets such that $|A_1|\geq\cdots\geq |A_n|$. Let $k\in\{1,\cdots, n\}$ and set $\sigma_k = \sum_{i>k} |A_i|$. Then we have
    $$|A_k|\geq \frac{|A_1|}{(k-1)!}-2(\sigma_k+1).$$
\end{lemma}
In our application we want to have as many colours as possible above the treshold $\frac{N}{n^{1.3}}$.
\begin{lemma}\label{lem:large_colour}
    Suppose that $\{1\cdots, N\} = A_1\cup\cdots\cup A_n$ is a partition into sum-free sets such that $|A_1|\geq\cdots\geq |A_n|$. There exists $k\gg \frac{\log n}{\log\log n}$ such that
    $$|A_k|\geq \frac{N}{n^{1.3}}.$$
\end{lemma}
\begin{proof}
    Let $k\geq c\frac{\log n}{\log \log n}$, for a small constant $c$ to be chosen so that $k$ is an integer and $k! \leq  n^{0.1}.$\\ Let us first assume that $|A_1|\leq \frac{N}{n^{0.1}}$. By Markov's inequality there are at least $n^{0.1}/2$ classes larger than $N/(2n)$, in particular $|A_k|\geq \frac{N}{n^{1.3}}$. On the other hand suppose $|A_1| > \frac{N}{n^{0.1}}$. If $\sigma_k = \sum_{i>k} |A_i| \leq \frac{N}{n^{0.3}}$, then by Lemma~\ref{lem:factorial} again $|A_k|\geq \frac{N}{n^{1.3}}$. If finally $\sigma_k > \frac{N}{n^{0.3}}$ then  $|A_k|\geq \sigma_k/n > \frac{N}{n^{1.3}}.$\\
\end{proof}
The next lemma allows us to consider one long equation instead of a couple of shorter ones.
\begin{lemma}\label{lem:extension}
    Suppose that a set of integers $A$ contains a solution to the equation
    $$x_1+x_2+\cdots+x_l = y_1+y_2+\cdots+y_r,$$
    where $l,r$ are positive integers. Then for any integers $t\geq 1, w\geq 0,$ $A$ contains a solution to the equation
    $$x_1+x_2+\cdots+x_{lt+w} = y_1+y_2+\cdots y_{rt+w}.$$
\end{lemma}
\begin{proof}
    Suppose that $x_1',x_2',\cdots, x_l', y_1',y_2',\cdots y_r'\in A$ form a solution to the first equation. Then we can choose
    $$x_i'' = x'_{\lfloor i/t\rfloor} \text{ for } i\leq lt,$$
    $$x_i'' = x'_1 \text{ for } i > lt,$$
    $$y_i'' = y'_{\lfloor i/t\rfloor} \text{ for } i\leq rt,$$
    $$y_i'' = x'_1 \text{ for } i > rt.$$
    That is clearly a solution to the second equation.
\end{proof}
We can use Lemma~\ref{lem:extension} to make a connection between cycles of odd length and a solution to a certain equation.
\begin{lemma}\label{lem:girth}
Suppose that $A\subseteq[N]$ is free from solutions to the equation $$x_1+x_2+\cdots+x_{12}=y_1+y_2+\cdots+y_{9}.$$
Let $G$ be a graph on $N$ vertices identified with $[N]$, where $u,v$ are connected if and only if $|u-v|\in A.$ Then $G$ does not contain any cycles of length $3$ or $5$.
\begin{proof}
We will argue that if $G$ contains a cycle of length $3$ or $5$, then there is a solution to the equation. Suppose that vertices $v_1,v_2,\cdots, v_5\in [N]$ presented in the increasing order form a cycle of length $5$. If $G$ contains a cycle of length $3$, we treat it as a cycle of length $5$, by going back and forth with the last edge.\\
Fix some orientation of this cycle and let $w_i$ be the successor of $v_i$ in this orientation. Of course $w_1,w_2,\cdots, w_5$ is a reordering of $v_1,v_2,\cdots, v_5$. Let $r_i$ be $1$ if $w_i>v_i$ and $-1$ otherwise.\\

\begin{figure}[H]
 \begin{tikzpicture}[
            > = stealth, % arrow head style
            shorten > = 1pt, % don't touch arrow head to node
            auto,
            node distance = 1.5cm, % distance between nodes
            semithick % line style
        ]

        \tikzstyle{every state}=[
            draw = black,
            thick,
            fill = white,
            minimum size = 4mm
        ]

        \node[state] (v1) {$v_1$};
        \node[state] (v2) [right of=v1] {$v_2$};
        \node[state] (v3) [right of=v2] {$v_3$};
        \node[state] (v4) [right of=v3] {$v_4$};
        \node[state] (v5) [right of=v4] {$v_5$};

        \draw [->] (v1) to [out=60,in=120] (v3);
        \draw [->] (v3) to [out=60,in=120] (v4);
        \draw [->] (v4) to [out=60,in=120] (v5);
        \draw [->] (v5) to [out=-140,in=-40] (v2);
        \draw [->] (v2) to [out=-120,in=-60] (v1);
    \end{tikzpicture}
    \caption{Here is an example of a cycle on 5 vertices. Edges connect vertices $v_i$ and $w_i$ in the orientation of the cycle. Arrows pointing right correspond to $r_i=1$ and arrows pointing left to $r_i=-1$.}
\end{figure}
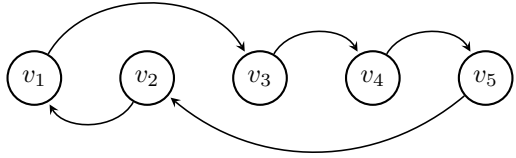
    
We have
$$\sum_{i=1}^5 |w_i-v_i| r_i = \sum_{i=1}^5 w_i-v_i  = \sum_{i=1}^5 v_i -\sum_{i=1}^5 v_i = 0.$$
Because $|w_i-v_i|\in A$ we have for some $a_1,a_2,\cdots, a_5\in A$ that
$$r_1a_1+r_2a_2+\cdots r_5a_5 = 0.$$
The number of $r_i$ equal to $1$ is at least $1$ and at most $4$, which means that one of these two situation holds
$$a_1' + a_2' +a_3' + a_4' = a_5',$$
$$a_1' + a_2' +a_3' = a_4'+a_5',$$
where $a_1', a_2',\cdots, a_5'$ is some reordering of $a_1,a_2,\cdots, a_5.$ By Lemma~\ref{lem:extension} any of these three equalities constitutes a solution to the equation $$x_1+x_2+\cdots+x_{12} = y_1+y_2+\cdots y_{9}$$.
\end{proof}
\end{lemma}
We are now ready to prove the next theorem.
\begin{proof}(of Theorem~\ref{thm:long})
Let $A_1,A_2,\cdots, A_n$ be the sets containing all numbers in each colour. Assume without loss of generality that $|A_1|\geq\cdots\geq |A_n|$. To make a further calculation easier we perform a technical clean up. 

Let us add three auxiliary colours $A'_1,A'_2, A'_3$ and use them to colour all numbers from the interval $[N+1, 2N]$. $A'_1$ will cover $[N+1, \lfloor 2^{1/3}N\rfloor]$, $A'_2$ will cover $[\lfloor 2^{1/3}N\rfloor +1, \lfloor 2^{2/3}N\rfloor]$ and $A'_3$ will cover the rest. This way we did not introduce any new solutions to the equation as $9\cdot 2^{1/3} < 12$. We now pretty much forget about the three additional colours, which is inconsequential because we will show that $N=O(\sqrt{(n+3-k)!})$, where $k\geq c\frac{\log n}{\log\log n}$. In return, each of the colours $A_1,A_2,\cdots, A_k$ introduces at least $N$ vertices of degree at least $|A_k|$ to the graph, we will use this observation later in the proof.

By Lemma~\ref{lem:large_colour}, there exists integer $k\geq c \frac{\log n}{\log\log n}$, where $c$ is a small constant, such that $|A_k|\geq \frac{N}{n^{1.3}}.$\\
For each $i\leq k$ we now construct graph $G_i$ on $2N$ vertices $v_1,v_2\cdots v_{2N}$ and edges between $v_{x}$ and $v_{y}$ if $|x-y|\in A_i$.\\
By Lemma~\ref{lem:girth}, graphs $G_1,G_2,\cdots, G_k$ are free from cycles of length $3$ and $5$. Moreover, the sets $A_1,A_2,\cdots, A_k$ contain no solutions to the equation
$$x_1+x_2+x_3+x_4=y_1+y_2+y_3.$$
Using these two observations we will find large independent sets in the graphs $G_1,G_2,\cdots, G_k.$ By a very similar reasoning as in the proof of Theorem~\ref{thm:schur_like} we argue that for each of these graphs the independence number is at least $d_3(v)/4$, where $v$ is an arbitrary vertex. For a fixed vertex $v_a$ we do so by considering four sets:
$$\{v_{a+ x+ y+ z} : x,y,z\in A_i, a+ x+ y+ z\in [2N]\},$$
$$\{v_{a+ x+ y- z} : x,y,z\in A_i, a+ x+ y- z\in [2N]\},$$
$$\{v_{a+ x- y- z} : x,y,z\in A_i, a+ x- y- z\in [2N]\},$$
$$\{v_{a- x- y- z} : x,y,z\in A_i, a- x- y- z\in [2N]\}.$$

Each of them will form an independent set because $(3A_i-3A_i)\cap A_i =\emptyset$ and the largest one will have size at least $d_3(v)/4$. Similarly, because we know that the graph is free of cycles of length 3 and we know that $\alpha(G_i)\geq d_1(v)$ and $\alpha(G_i)\geq d_2(v)$.
Knowing this, we apply Lemma~\ref{lem:shearer2} to $G_i$ with $m=3$ to find independent sets $S_1,S_2,\cdots S_k$, of sizes $\alpha(G_1), \alpha(G_2),\cdots, \alpha(G_k)$, where  
$$\alpha(G_i)\geq \sum_{v\in G_i}\Bigl[\frac{d_1(v)/2}{1+\alpha(G)+\alpha(G)+\alpha(G)/4}\Bigr]^{1/2}.$$
Thus we get
$$\alpha(G_i)^{3/2}\geq \frac{1}{2}\sum_{v\in G_i}d_1(v)^{1/2}\geq \frac{N}{2}\frac{\sqrt{N}}{n^{0.65}}.$$
And so $\alpha(G_i)\gg \frac{N}{n^{0.45}}$. Since each $S_i$ can be looked at as a subset of $[2N]$, we can write
$$\sum_{t_1,t_2,\cdots, t_k\in [-2N, 2N]} |(t_1+S_1)\cap (t_2+S_2)\cap \cdots \cap (t_k+S_k)\cap [2N]|\geq |S_1||S_2|\cdots |S_n| 2N,$$
and so for some choice of the shifts $t_1',t_2',\cdots, t_k'$ the intersection is at least as big as the average, so that
$$|(t_1'+S_1)\cap (t_2'+S_2)\cap \cdots \cap (t_k'+S_k)\cap [2N]|\geq \frac{1}{2^k n^{0.45k}}2N\gg \frac{1}{n^{0.46k}}N.$$
Call the above intersection $S$, let $s_1, s_2,\cdots, s_q$ be all its elements in the increasing order.\\It is easy to argue that the set $X:=\{s_2-s_1,s_3-s_1,\cdots, s_q-s_1\}$ is free from colours $A_1,A_2,\cdots, A_k$. If not, then $s_j-s_1\in A_i$ for some $i\leq k$, so there is an edge between $s_j-t_i'$ and $s_1-t_i'$ in the graph $G_i$. But $S\subseteq t_i'+S_i$, so $s_j-t_i'$ and $s_1-t_i'$ are elements of $S_i$. We have a contradiction as we found an edge between two vertices of an independent set.\\
Let us continue the induction on this set for $n-k$ steps as in the proof of Theorem~\ref{thm:schur_like} to construct sets $X\supseteq X_{1}\supseteq\cdots\supseteq X_{n-k}$ with $$|X_{i+1}|\geq \frac{|X_i|}{3\sqrt{n-k+1-i}}.$$ 
The last set $X_{n-k}$ must have size at most 1 so $|X|=O( \sqrt{(n-k)!})$. That means $N=O(n^{0.46k}|X|)=O(\sqrt{(n-k/100)!})$.
\end{proof}
If we could assume that $|A_1|,|A_2|,\cdots, |A_n|\gg \frac{N}{n}$ in the above argument, we could apply Lemma~\ref{lem:shearer2} on all colours at once and obtain the bound $N=O((n!)^{1/3})$, and choosing a suitably long equation for a given $m$ we would have $N=O((n!)^{1/m}).$
\section{Other non-invariant equations}
Theorem~\ref{thm:schur_like} and Lemma~\ref{lem:extension} immediately imply that the same upper bound $O(\sqrt{n!})$ holds for all equations
$$x_1+x_2+\cdots+x_{3t+w} = y_1+y_2+\cdots y_{2t+w},$$
where $t\geq 1, w\geq 0$ are integers. This leaves out quite a number of equations, even if we only consider those with coefficients $\pm1$, the first interesting case being $x_1+x_2+x_3=y$. It is worth mentioning that in a graph constructed from a colouring of $[N]$, already introduced in the proofs of the two previous theorems, every cycle of length $5$, corresponds either to such equation or the equation from Theorem~\ref{thm:schur_like}. 
It is however not clear how sets free of solutions to such equation should behave, as proofs of Schur's Theorem do not generalize in that case.
We can however give a similar bound based on the observation that in any $n$-colouring of $[\lfloor(2n)!\rfloor]$ there is an abundance of monochromatic Schur triples.
\begin{theorem}\label{thm:three}
Let $l\geq 3, N$ be positive integers and suppose that each of the numbers $1,2,\cdots, N$ has one of $n$ colours assigned to it. If that there are no monochromatic solutions to the equation $x_1+x_2+x_3=y$. Then we have $N\leq e(2n)!$.
\end{theorem}
\begin{proof}
Suppose for contradiction $N > e(2n)!$ and a colouring $A_1,A_2,\cdots, A_n$ is given, which avoids monochromatic solutions to our equation. We also create $n$ additional colours $A_1',A_2',\cdots, A_n'$, which are initially empty. From Schur's theorem, we know that whenever we colour $[N]$ with $2n$ colours we will find a monochromatic sum. Suppose that $a+b=c$ and $a,b,c\in A_i$. Let us change the colour of $c$ to $A_i'$ and repeat the process until we find a monochromatic sum in $A_i'$ for some $i$. This must happen after at most $N$ steps. Suppose that $a'+b'=c'$ and $a',b',c'\in A_i'$. That means in the original colouring we had $x_1+x_2=a'$ with $x_1,x_2,a'\in A_j$ and so there was a monochromatic solution $x_1+x_2+b'=c'$ inside the colour $A_i.$ 
\end{proof}
\begin{figure}[H]
 \begin{tikzpicture}[
            > = stealth, % arrow head style
            shorten > = 1pt, % don't touch arrow head to node
            auto,
            node distance = 1.5cm, % distance between nodes
            semithick % line style
        ]

        \tikzstyle{every state}=[
            draw = black,
            thick,
            fill = white,
            minimum size = 4mm
        ]

        \node[state] (v1) {$c'$};
        \node[state] (v2) [below left of=v1] {$a'$};
        \node[state] (v3) [below right of=v1] {$b'$};
        \node[state] (v4) [below left of=v2] {$x_1$};
        \node[state] (v5) [below right of=v2] {$x_2$};

        \draw [->] (v2) to (v1);
        \draw [->] (v3) to (v1);
        \draw [->] (v4) to (v2);
        \draw [->] (v5) to (v2);
    \end{tikzpicture}
    
    \caption{It is helpful to think about the variables from the proof as forming a tree, with $c'$ in the root, the second layer are elements of colours $A_1',A_2',\cdots, A_n'$ and the third layer elements of colours $A_1',A_2',\cdots, A_n'$.}
\end{figure}
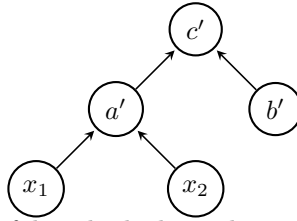
Proof of Theorem~\ref{thm:imbalanced} follows pretty much in the same way.
\begin{proof}(of Theorem~\ref{thm:imbalanced})\\
Let $r=\lceil\log_2 l\rceil$ and for contradiction assume that $N>e(rn)!$ and a colouring $A_1,A_2,\cdots, A_n$ is given, which avoids monochromatic solutions to our equation. For every colour $A_i$ we construct $r-1$ initially empty auxiliary colours $A_i^2, A_i^3,\cdots, A_i^{r}$ and set $A_i^1=A_i$. Since there are $rn$ colours in total there will always be a monochromatic sum. Suppose that at some point of our procedure we have $a,b,a+b\in A_i^j$. Then we move $a+b$ from $A_i^j$ to $A_i^{j+1}$. Ultimately, there will be $a,b,a+b\in A_i^{r}$ for some~$i$. We argue by induction on $r$ that any number $c\in A_i^{r}$ can be expressed as a sum of $k\leq 2^{r-1}$ numbers from $A_i^1\cup A_i^2\cup \cdots\cup A_i^{r}$ and thus $a+b\in A_i^{r}$ can be expressed as a sum of $l\leq 2^r$ numbers of the same initial colour $A_i$. If $r=1$ the assertion is trivial as $c=c$. If $r>1$ let $k'=\lfloor k/2\rfloor$ and $k''=\lceil k/2 \rceil$. Clearly $k'\leq k''\leq 2^{r-1}$ and since $c\in A_i^{r}$ there exist numbers $c',c''\in A_i^{r-1}$ such that $c'+c''=c$. By induction hypothesis they are expressed as a sum of $k'$ and $k''$ numbers from $A_i^1\cup A_i^2\cup \cdots\cup A_i^{r-1}$ respectively. Thus $c$ is expressed as a sum of $k'+k''=k$ numbers from $A_i^1\cup A_i^2\cup \cdots\cup A_i^{r}$. We have a contradiction as that means the sum of $l$ numbers from $A_i$ also belongs to $A_i$. 
\end{proof}
Finally, we can implement a similar idea to give a bound on the equation with coefficients different than $-1,0,1$, namely $$x_1+2x_2+x_3=y_1+y_2.$$
\begin{proof}(Of Theorem~\ref{thm:coefficients})
We proceed in the same way as in the proof of Theorem~\ref{thm:three} by creating $n$ additional colours $A_1',A_2',\cdots, A_n'$. This time whenever $a,b,c\in A_i$ and $a+b=c$ we change the colour of $a$ to $A_i'$. After at most $N$ steps we find $a'+b'=c'$ with $a',b',c'\in A_j'$. That means there exist $x_1,y_1\in A_j$ such that $x_1+a'=y_1$. Our monochromatic solution is then $$x_1+2a'+b'=y_1+c'.$$
\end{proof}
\begin{figure}[H]
 \begin{tikzpicture}[
            > = stealth, % arrow head style
            shorten > = 1pt, % don't touch arrow head to node
            auto,
            node distance = 1.5cm, % distance between nodes
            semithick % line style
        ]

        \tikzstyle{every state}=[
            draw = black,
            thick,
            fill = white,
            minimum size = 4mm
        ]

        \node[state] (v1) {$c'$};
        \node[state] (v2) [below left of=v1] {$a'$};
        \node[state] (v3) [below right of=v1] {$b'$};
        \node[state] (v4) [above left of=v2] {$y_1$};
        \node[state] (v5) [below left of=v4] {$x_1$};

        \draw [->] (v2) to (v1);
        \draw [->] (v3) to (v1);
        \draw [->] (v2) to (v4);
        \draw [->] (v5) to (v4);
    \end{tikzpicture}
\end{figure}

\textsc{Faculty of Mathematics and Computer Science, Adam Mickiewicz University, Umultowska 87, 61-614 Poznan, Poland }

\textit{Email address:} tomasz.kosciuszko@amu.edu.pl
\end{document}